\documentclass{article}
\title{An incomplete variant of Wilson's congruence}
\author{Joel Beeren \\
School of Mathematics and Statistics\\ 
University of New South Wales, NSW 2052, Australia\\
joel.b@unsw.edu.au
\and
David Harvey\footnote{Supported by Australian Research Council DECRA Grant DE120101293.}\\ School of Mathematics and Statistics\\ University of New South Wales, NSW 2052, Australia\\
d.harvey@unsw.edu.au
\and
Tim Trudgian\footnote{Supported by Australian Research Council DECRA Grant DE120100173.}\\
Mathematical Sciences Institute\\ The Australian National University,
 ACT 0200, Australia\\ timothy.trudgian@anu.edu.au
}
\usepackage{url}
\usepackage{amsthm}
\usepackage{amsmath}
\usepackage{amssymb}
\usepackage{booktabs}

\newtheorem{Lem}{Lemma}
\newtheorem{con}{Conjecture}
\newcommand{\ZZ}{\mathbb Z}
\newcommand{\QQ}{\mathbb Q}
\begin{document}

\maketitle
\begin{abstract}
\noindent
This article examines the nontrivial solutions of the congruence \[ (p-1)\cdots(p-r) \equiv -1 \pmod p. \] We discuss heuristics for the proportion of primes $p$ that have exactly $N$ solutions to this congruence. We supply numerical evidence in favour of these conjectures, and discuss the algorithms used in our calculations.
\end{abstract}

\section{Heuristics and conjectures}
Wilson's Theorem \cite[Theorems 80 and 81]{HW} states that
 \[ (p-1)! \equiv -1 \pmod p \]
if and only if $p$ is a prime. Now truncate the factorial after $r$ terms. For which primes $p$ is there an $r$ for which
\begin{equation}\label{wit}
 (p-1)\cdots(p-r) \equiv -1 \pmod p ?
\end{equation}
Certainly $r=1$ is trivial; $r=p-1$ follows from Wilson's Theorem, whence $r=p-2$ follows trivially. Henceforth we consider only $2 \leq r \leq p - 3$.

Initially we proceeded as follows. The congruence \eqref{wit} has no solutions if and only if none of the $p - 4$ integers $(p-1)\cdots(p-r) + 1$, for $2\leq r \leq p-3$, are divisible by $p$. The probability that a prime $p$ divides a `random' integer $N$ is $1/p$. Given $m$ random integers chosen independently, the probability that $p$ does not divide any of them is then $(1 - 1/p)^m$. Thus, under heroic randomness and independence assumptions, we expect the proportion of $p$ for which \eqref{wit} has no solutions to be roughly
 \[ (1-1/p)^{p-4}\rightarrow e^{-1} \approx 0.36788 \]
when $p$ is large.

Turning to numerical experiment, we find that 429 of the 1229 primes less than $10^4$ have no solutions to \eqref{wit}. The proportion is 0.349, reasonably close to our initial guess.

It turns out that this guess is almost certainly wrong. The remainder of this paper may serve as yet another cautionary tale about the dangers of heuristic probabilistic reasoning in number theory.

First, our independence assumption is not justified. Denoting by $T_r$ the partial product $T_r = (p - 1) \cdots (p - r)$, we have:
\begin{Lem}\label{Lemma1}
For any $2 \leq r \leq p - 3$,
 \[ T_r T_{p-r-1} \equiv (-1)^{r+1} \pmod p. \]
\end{Lem}
\begin{proof}
Observe that $T_r \equiv (-1)^r r! \pmod p$, and apply Wilson's Theorem.
\end{proof}
Thus the cases $r$ and $s = p - r - 1$ are not independent. For odd $r$, we see that \eqref{wit} holds for $r$ if and only if it holds for $s$. For even $r$, we see that \eqref{wit} holds for either $r$ or $s$, but not both; and it holds for one of them if and only if $T_r \equiv \pm 1 \pmod p$.

Taking these observations into account, we should posit $p/4 + O(1)$ independent events with probability $1 - 1/p$ corresponding to the odd $r < p/2$, and $p/4 + O(1)$ independent events with probability $1 - 2/p$ corresponding to the even $r < p/2$. Our revised estimate for the proportion of primes for which \eqref{wit} has no solutions is thus
 \[ (1 - 1/p)^{p/4} (1 - 2/p)^{p/4} \rightarrow e^{-3/4} \approx 0.47237. \]
This `improved' heuristic is an even worse match for the observed data!

In fact we are on the right track. We have simply forgotten the following arithmetic gem.
\begin{Lem}\label{Lemma2}
Let $p \equiv 3 \pmod 4$. Then
\begin{equation*}
\left(\frac{p-1}2\right)! \equiv (-1)^{\nu_p} \pmod p,
\end{equation*}
where $\nu_p$ is the number of quadratic non-residues $1 < x < p/2$.
\end{Lem}
\begin{proof}
See \cite[Theorem 114]{HW}.
\end{proof}
When $\nu_p$ is even, the congruence \eqref{wit} automatically has the solution $r = (p-1)/2$. To incorporate this into our model, we must address the question as to how often $\nu_p$ is even. Numerical evidence (see Section \ref{sec:computations}) suggests the following conjecture:
\begin{con}\label{con1}
For $p \equiv 3 \pmod 4$, the proportion of $p$ for which $\left(\frac{p-1}2\right)! \equiv 1 \pmod p$ approaches $\frac{1}{2}$ as $p \to \infty$.
\end{con}
We are not aware of this conjecture having appeared before in print, but it has been raised on the Mathoverflow discussion forum \cite{MO}. The problem has been recast by Mordell \cite{Mordell1961} in terms of the class number $h(-p)$ of $\QQ(\sqrt{-p})$. Namely, for $p > 3$ we have
 \[ \nu_p = \begin{cases} 0 \pmod 2 & \text{if $h(-p) \equiv 3 \pmod 4$}, \\ 1 \pmod 2 & \text{if $h(-p) \equiv 1 \pmod 4$}. \end{cases} \]
We do not know if this interpretation sheds any light on Conjecture \ref{con1}.

We now revise our model a second time, taking into account Lemma \ref{Lemma2} and Conjecture \ref{con1}. Of those primes satisfying $p \equiv 1 \pmod 4$, asymptotically half the primes, our estimate for the proportion of primes for which \eqref{wit} has no solution is still $e^{-3/4}$. For those primes $p \equiv 3 \pmod 4$ with $\nu_p$ odd, the estimate is again $e^{-3/4}$. According to Conjecture \ref{con1} this accounts for another quarter of the primes. However, for the remaining primes, where $\nu_p$ is even, our estimate is zero. This leads to our main conjecture.
\begin{con}\label{Conmain}
The proportion of primes for which \eqref{wit} has no nontrivial solutions is
 \[ \frac{3}{4} e^{-3/4} \approx 0.3542749. \] 
\end{con}

Using the same model, we may develop a more refined conjecture that estimates the proportion of primes $p$ for which there are exactly $N$ values of $r$ satisfying \eqref{wit}.
\begin{con}\label{ConN}
Let $N \geq 0$. The proportion of primes $p$ for which \eqref{wit} has exactly $N$ nontrivial solutions is
 \[ 
\frac{e^{-3/4}}{2^{N+1}}
 \left( \frac{3}{2} \sum_{k=0}^{\lfloor N/2 \rfloor} \frac{1}{k! (N-2k)!} + \sum_{k=0}^{\lfloor (N-1)/2 \rfloor} \frac{1}{k! (N-1-2k)!}\right).
 \]
\end{con}
This formula is derived as follows. For $k \geq 0$, denote by $P_k$ the probability that \eqref{wit} has exactly $k$ odd solutions in the range $3 \leq r < (p-1)/2$. By the discussion following Lemma \ref{Lemma1}, and the usual properties of the binomial distribution, for large $p$ our model suggests that
 \[ P_k = \binom{p/4 + O(1)}{k} \left(\frac1p\right)^k \left(1 - \frac1p\right)^{p/4 - k + O(1)} \rightarrow \frac{e^{-1/4}}{4^k k!}. \]
Similarly, for $\ell \geq 0$ denote by $Q_\ell$ the probability that $T_r \equiv \pm 1 \pmod p$ has exactly $\ell$ even solutions in the range $2 \leq r < (p-1)/2$. Then
 \[ Q_\ell = \binom{p/4 + O(1)}{\ell} \left(\frac2p\right)^\ell \left(1 - \frac2p\right)^{p/4 - \ell + O(1)} \rightarrow \frac{e^{-1/2}}{2^\ell \ell!}. \]
Assuming that the behaviour for odd and even $r$ is independent, the probability of observing exactly $N$ solutions for $2 \leq r \leq p - 3$, $r \neq (p-1)/2$, should be
 \[ \sum_{2k + \ell = N} P_k Q_\ell = \sum_{k=0}^{\lfloor N/2\rfloor} \frac{e^{-1/4}}{4^k k!} \cdot \frac{e^{-1/2}}{2^{N-2k} (N-2k)!} = \frac{e^{-3/4}}{2^N} \sum_{k=0}^{\lfloor N/2\rfloor} \frac1{k!(N-2k)!}. \] 
Finally, for $p \equiv 1 \pmod 4$, and for $p \equiv 3 \pmod 4$ with $\nu_p$ odd, the probability that \eqref{wit} has exactly $N$ solutions is given by the above formula (the exceptional value $r = (p-1)/2$ makes a negligible contribution asymptotically). For $p \equiv 3 \pmod 4$ with $\nu_p$ even, we must replace $N$ by $N-1$ to account for the automatic solution $r = (p-1)/2$. Our final estimated probability is thus
 \[ \frac34 \left(\frac{e^{-3/4}}{2^N} \sum_{k=0}^{\lfloor N/2\rfloor} \frac1{k!(N-2k)!} \right) + \frac 14 \left(\frac{e^{-3/4}}{2^{N-1}} \sum_{k=0}^{\lfloor (N-1)/2\rfloor} \frac1{k!(N-1-2k)!}\right). \]

\section{Algorithms and computations}
\label{sec:computations}

We first consider the motivating problem, counting the number of nontrivial solutions to \eqref{wit}. For this the na\"ive algorithm appears to be the best available. For each prime $p$ up to some bound, we compute $T_2, T_3, \ldots$, by successive multiplication modulo $p$, and count how many times we see $-1$.

We wrote a simple C implementation of this algorithm, paying some attention to efficient modular arithmetic. We ran it for all primes up to $10^8$. The running time was 22 hours on a 16-core 2.6 GHz Intel Xeon server. Table \ref{tab1} summarises the results. The last column shows the probabilities for each $N$ proposed in Conjecture \ref{ConN}; they are a superb fit for the observed proportions in the previous column.

It is difficult to push the search bound higher. The running time for each prime is essentially linear in $p$, so the cost of handling all $p < x$ grows essentially quadratically in $x$. For example, to increase the search bound to $10^9$ would take about three months on the same hardware. We do not know of any asymptotically faster algorithms for this problem.
\begin{table}[h]
\centering
\caption{Statistics of nontrivial solutions to \eqref{wit} for $p < 10^8$}
\begin{tabular}{rrrr}
\toprule
 $N$  & \# Primes with $N$ solutions   & Proportion & Conjecture \ref{ConN} \\
\midrule
0           &   2041117      &  0.3542711  &   0.3542749   \\
1           &   1701240      &  0.2952796  &   0.2952291   \\
2           &   1104376      &  0.1916835  &   0.1918989   \\
3           &    553921      &  0.0961426  &   0.0959495   \\
4           &    232308      &  0.0403211  &   0.0402865   \\
5           &     87019      &  0.0151037  &   0.0151612   \\
6           &     29037      &  0.0050399  &   0.0050358   \\
7           &      8887      &  0.0015425  &   0.0015638   \\
8           &      2631      &  0.0004567  &   0.0004423   \\
9           &       692      &  0.0001201  &   0.0001190   \\
10          &       165      &  0.0000286  &   0.0000298   \\
11          &        42      &  0.0000073  &   0.0000071   \\
12          &        17      &  0.0000030  &   0.0000016   \\
13          &         3      &  0.0000005  &   0.0000004   \\
\midrule
Total       &   5761455      &  1.0000000  &   1.0000000   \\
\bottomrule
\end{tabular}
\label{tab1}
\end{table}

Next we consider the problem of computing $((p-1)/2)! \pmod p$ for $p \equiv 3 \pmod 4$, in order to test Conjecture \ref{con1}. For this there is a greater variety of algorithms available. The na\"ive approach leads to an $O(p)$ algorithm as above (with a comparable implied constant). An algorithm with complexity $p^{1/2+o(1)}$ can be deduced from \cite{BGS}. We opted to implement an algorithm with average complexity only $(\log p)^{4 + o(1)}$ per prime, using the ``accumulating remainder tree'' technique introduced in \cite{CGH}.

We give a brief sketch of this algorithm. Suppose that we wish to compute $r_p = ((p-1)/2)! \bmod p$ for all primes $p \equiv 3 \pmod 4$ in some interval $2M < p < 2N$, where $M$ and $N$ are positive integers. Consider the binary tree, with nodes indexed by pairs $(a, b)$, where $b > a > 0$ are integers, defined as follows. The root node is $(M, N)$. The children of a given node $(a, b)$ are $(a, c)$ and $(c, b)$, where $c = \lfloor (a + b)/2\rfloor$. For each node let
 \[ I_{a,b} = \{k \in \ZZ : \text{$k$ odd, } 2a < k < 2b\}. \]
Thus at level $d$, the intervals $I_{a,b}$ partition $I_{M,N}$ into $2^d$ subintervals of roughly equal size. We stop at level $\ell = \lceil \log_2(N - M)\rceil$; at this level each $I_{a,b}$ has cardinality either zero or one.

The algorithm now proceeds as follows. First, for each node let
 \[ P_{a,b} = \prod_{\substack{p \in I_{a,b} \\ p \equiv 3 \bmod 4 \\ \text{$p$ prime}}} p, \qquad  V_{a,b} = \prod_{k \in I_{a,b}} \frac{k+1}2. \]
Compute $V_{a,b}$ and $P_{a,b}$ for each node, working from the bottom of the tree to the top, using the identities $V_{a,b} = V_{a,c} V_{c,b}$ and $P_{a,b} = P_{a,c} P_{c,b}$. Second, for each node let
 \[ X_{a,b} = a! \bmod {P_{a,b}}. \]
Compute $X_{M,N} = M! \bmod P_{M,N}$ using the method of Sch\"onhage (see for example \cite[Prop.~2.3]{CGH}). Then compute $X_{a,b}$ for each node, now working from the top of the tree downwards, using the formulae $X_{a,c} = X_{a,b} \bmod P_{a,c}$ and $X_{c,b} = X_{a,b} V_{a,c} \bmod P_{c,b}$ to descend from each node to its children. Finally, for each $p \equiv 3 \pmod 4$ in the interval $2M < p < 2N$, there is a unique node $(a, b)$ at level $\ell$ such that $p \in I_{a,b}$; for this node we have $I_{a,b} = \{p\}$, $P_{a,b} = p$ and $X_{a,b} = ((p-1)/2)! \pmod p$. For more details, including a complexity analysis, see \cite{CGH}. We mention here only that the complexity bound depends essentially on asymptotically fast algorithms for multiplication and division of large integers.

Using a straightforward implementation of the above algorithm in the Sage computer algebra system \cite{sage}, we computed $r_p$ for all $p < 10^{10}$. To keep memory usage under control, we split the work into intervals $(M, N)$ of size $1.5 \times 10^8$. The total CPU time expended was 4.4 days. The results, shown in Table \ref{tab2}, are in excellent agreement with Conjecture \ref{con1}.
\begin{table}[h]
\centering
\caption{Statistics of $r_p$ for $p < 10^{10}$, $p \equiv 3 \pmod 4$}
\begin{tabular}{lrrr}
\toprule
 $X$  & $\# \{p < X\}$ & $\# \{ p < X: r_p = 1 \}$ & proportion \\
\midrule
$10^1$     &             2  &              1   &  0.5000000000 \\
$10^2$     &            13  &              6   &  0.4615384615 \\
$10^3$     &            87  &             43   &  0.4942528736 \\
$10^4$     &           619  &            310   &  0.5008077544 \\
$10^5$     &        4\,808  &         2\,418   &  0.5029118136 \\
$10^6$     &       39\,322  &        19\,704   &  0.5010935354 \\
$10^7$     &      332\,398  &       166\,270   &  0.5002135994 \\
$10^8$     &   2\,880\,950  &    1\,440\,268   &  0.4999281487 \\
$10^9$     &  25\,424\,042  &   12\,713\,329   &  0.5000514474 \\
$10^{10}$  & 227\,529\,235  &  113\,772\,462   &  0.5000344769 \\
\bottomrule
\end{tabular}
\label{tab2}
\end{table}

\section*{Acknowledgements}

We are indebted to Professor Roger Heath-Brown for suggesting part of the argument leading to Conjectures \ref{Conmain} and \ref{ConN}.

\bibliographystyle{plain}
\bibliography{IncompleteWilson}

\end{document}